\documentclass{amsart}
\usepackage{amssymb,amsmath,amsthm,amsfonts}
\usepackage{graphicx}
\usepackage{enumitem}   
\usepackage{xcolor}    

\newcommand{\field}[1]{\mathbb{#1}}
\newcommand{\N}{\field{N}}                      
\newcommand{\R}{\field{R}}                      

\newcommand{\de}{\delta}

\newcommand{\ovdimB}{{\overline{\dim_B}}}
\newcommand{\Dcal}{\mathcal{D}}

\newcommand{\Ctil}{\widetilde{C}}

\def\Barint_#1{\mathchoice
          {\mathop{\vrule width 6pt height 3 pt depth -2.5pt
                  \kern -8pt \intop}\nolimits_{#1}}%
          {\mathop{\vrule width 5pt height 3 pt depth -2.6pt
                  \kern -6pt \intop}\nolimits_{#1}}%
          {\mathop{\vrule width 5pt height 3 pt depth -2.6pt
                  \kern -6pt \intop}\nolimits_{#1}}%
          {\mathop{\vrule width 5pt height 3 pt depth -2.6pt
                  \kern -6pt \intop}\nolimits_{#1}}}

\theoremstyle{definition}
\newtheorem{theorem}{Theorem}
\newtheorem{theoremA}{Theorem}

\theoremstyle{definition}

\newtheorem{remarks}[theorem]{Remarks}

\numberwithin{theorem}{section} \numberwithin{equation}{section}

\title{Dimensions and metric dyadic cubes}
\author[Efstathios-K. Chrontsios-Garitsis]{Efstathios-K. Chrontsios-Garitsis}
\address{Department of Mathematics \\ University of Tennessee, Knoxville \\ 1403 Circle Dr \\ Knoxville, TN 37966}
\email{echronts@utk.edu, echronts@gmail.com}

\begin{document}

\begin{abstract}
In this note, we provide equivalent definitions for fractal geometric dimensions through dyadic cube constructions. Given a metric space $X$ with finite Assouad dimension, i.e., satisfying the doubling property, we show that the construction of systems of dyadic cubes by Hyt\"onen-Kairema is compatible with many dimensions. In particular, the Hausdorff, Minkowski, and Assouad dimensions can be equivalently expressed solely using  dyadic cubes in the aforementioned system. The same is true for the Assouad spectrum, a collection of dimensions introduced by Fraser-Yu.
\end{abstract}

\maketitle

\section{Introduction}
For the past three decades, fractals have gained significant attention within both pure and applied mathematics. The study of dimension notions is a fundamental part of fractal geometry, which facilitates the understanding of fractal objects. Two notions that have been popular since the early 20th century are the Hausdorff and the Minkowski dimension, which have found applications within many fields of research. For instance, these dimensions often show up in the context of partial differential equations (PDEs) \cite{MinkNavierStokesEq}, number theory \cite{dimH_Number_Th}, signal processing \cite{MinkSignal} and mathematical physics \cite{MinkPhysics}. We refer to the book of Falconer \cite{FalcBook} for a thorough exposition on these dimensions. Another notion of dimension that has been enjoying a lot of interest for the past decade is the Assouad dimension. Initially introduced under a different name by Assouad in \cite{Assouad83}, in order to investigate embeddability properties of abstract metric spaces, the Assouad dimension has been recently tied to a lot of areas through its fractal-geometric properties. Moreover, this notion has even motivated the introduction of dimension spectra, such as the Assouad spectrum introduced by Fraser-Yu \cite{fy:assouad-spectrum}, which interpolate between known dimensions (see \cite{FraserBook} for a modern exposition and applications). 

All three  notions of dimension have proved to be extremely useful, and are often studied simultaneously in situations where they differ. However, their original definitions can be quite complicated for the sake of calculations. The main idea behind the Hausdorff dimension $\dim_H E$ of a set $E\subset \R^n$ is that the dimension of the set is essentially the critical exponent $s\geq 0$ for which the $s$-dimensional Hausdorff measure of $E$, denoted by $H^s(E)$, changes from being infinite to being $0$. This is motivated by the property of the Lebesgue measure that, for instance, assigns infinite length and zero volume to a $2$-dimensional object. On the other hand, the Minkowski dimension $\dim_B E$ of a bounded set $E\subset \R^n$ is intuitively the exponent $s>0$ for which $N(E,r)$, i.e. the minimum number of sets of diameter at most $r$ needed to cover $E$, is approximately $r^{-s}$, for all small $r>0$. The Assouad dimension $\dim_A E$ of $E$ is defined by utilizing the idea of the Minkowski dimension after ``zooming in" (or ``zooming out") at all points of $E$. 
For $\theta\in (0,1)$, the $\theta$-Assouad spectrum $\dim_A^\theta E$ of $E$ is defined similarly to the Assouad dimension, with a restriction on the allowed zoom scales that depends on $\theta$.
We refer to Section \ref{sec:background} for rigorous definitions.

In the Euclidean setting, one feature of $\R^n$ that has facilitated the study of dimensions, especially within applied areas, has been the simplification of the aforementioned definitions using dyadic cubes. In particular, the Hausdorff dimension of $E$ is equal to the critical exponent $s\geq 0$ for which the $s$-dimensional cubic measure $M^s(E)$ changes from being infinite to being equal to $0$. The cubic measure $M^s(E)$ is defined using dyadic cubes, unlike the Hausdorff measure which uses arbitrary sets. Similarly, for the definition of the Minkowski and Assouad dimension, if $E$ is contained in a ball $B(x,R)$ and its circumscribed cube $Q(x,R)$, one can replace the number $N(E,r)$ by the smallest number of dyadic sub-cubes of $Q(x,R)$ of level $m$ needed to cover $E$, denoted by $D(E,m)$, and ask that $D(E,m)\simeq 2^{m s}$, for all large $m\in \N$. 
Such simplifications have been an invaluable tool in many applications of fractal geometry. 
In particular, in the area of dimension distortion under mapping classes, the simplified definitions have been used extensively (for instance \cite{GehringVais, Kaufman, SobSubcriticalDimDist1}), including in previous work of the author in \cite{Chron_Sob, HolomSpecChron}, and jointly by Tyson and the author in \cite{OurQCspec, HolderSpecOur}.

Given the purely metric nature of the Hausdorff dimension, Minkowski dimension, Assouad dimension, and Assouad spectrum, it would be helpful to have similar reductions of the original definitions in the setting of more abstract spaces. In fact, the origin and popularity of the Assouad dimension within the analysis on metric spaces community \cite{Assouad83, hei:lectures}, and the recent embeddability theorem of Troscheit and the author \cite{ChronSascha} emphasize this need for simplified definitions in higher generality. In arbitrary metric spaces, however, there are various generalizations of dyadic cube constructions. One of the first manuscripts addressing this idea was by David  \cite{CubesGuyC1}, while one of the first explicit constructions of a system of dyadic cubes is due to Christ \cite{ChristCubes}. See also \cite{MoreCubes1},  \cite{MoreCubes2}, \cite{MoreCubes3}, which is not an exhaustive list. 

The main result of this note is to express the aforementioned dimensions equivalently in the metric spaces setting using the dyadic cube systems constructed by Hyt\"onen-Kairema in \cite{Hyt:dyadic}. Specifically, we show that the Hausdorff measure $H^s(E)$ and minimal covering number $N(E,r)$ 
can be replaced by the cubic measure $M^s(E)$ and dyadic cube covering number $D(E,m)$, respectively, as defined through the Hyt\"onen-Kairema dyadic cube systems (see Section \ref{sec:background} for  definitions).

\begin{theorem}\label{thm:main}
    Let $(X,d)$ be a doubling metric space. There is a finite collection of $\delta$-dyadic cube systems $\mathcal{S}=\{\Dcal_1, \dots, \Dcal_K\}$ of $X$, such that for every non-empty $E\subset X$, the following hold:
    \begin{enumerate}[label=(\roman*)]
        \item For any $t\in \{1,\dots, K\}$, we have
        $$
        \dim_H E= \inf \{ s \geq 0 : M^s(E)=0 \}=\sup \{ s \geq 0 : M^s(E)=\infty \},
        $$ where $M^s(E)$ is the cubic measure of $E$ with respect to $\Dcal_t$. 

        \item If there are $x\in X$, $R>0$ with $E\subset B(x,R)$, and $m_E\in \N$ with $\delta^{m_E}\leq |E|$, then 
        $$
        \dim_B E = \inf\{ \alpha\geq 0: \exists \, C>0\,\, \text{s.t.}\,\, D(E,m)\leq C \delta^{-m\alpha}\,\, \text{for all}\,\, m\geq m_E\},
        $$ where $D(E,m)$ is with respect to the circumscribed cube $Q(x,R)$  of the ball $B(x,R)$, which lies in $\Dcal_t$ for some $t\in \{1,\dots, K\}$.
\newpage
        \item There is a constant $\Ctil>0$, independent of $E$, such that for any $\theta\in (0,1)$, the Assouad spectrum $\dim_A^\theta E$ is equal to the infimum of all $\alpha>0$ for which there is $C>0$ with
        $$
        D(E\cap B(x,R),m) \le C \delta^{-m\alpha},
        $$ for all $x\in E$, $R>0, m\in \N$ with $0<\Ctil \delta^{m}R\le R^{1/\theta} < R < 1$, and $D(E\cap B(x,R),m)$ with respect to the circumscribed cube $Q(x,R)$.
        

        \item 
        %
        There is a constant $\Ctil>0$, independent of $E$, such that the Assouad dimension $\dim_A E$ is equal to the infimum of all $\alpha>0$ for which there is $C>0$ with
        $$
        D(E\cap B(x,R),m) \le C \delta^{-m\alpha},
        $$ for all $x\in E$, $R>0$,  $m\in \N$ with $\Ctil \delta^{m}R\leq R$, and $D(E\cap B(x,R),m)$ with respect to the circumscribed cube $Q(x,R)$.
    \end{enumerate}
\end{theorem}

This paper is organized as follows. Section \ref{sec:background} reviews the background on the relevant dimensions and introduces the quantities $M^s(E)$, $D(E,m)$ for $E\subset X$, using the dyadic cube systems of Hyt\"onen-Kairema for the  metric space $X$.  In Section \ref{sec: Proofs} we prove our main result, Theorem \ref{thm:main}, by employing properties of the dyadic cubes of Hyt\"onen-Kairema that resemble those of the Eucledian systems. Section \ref{sec:Final Remarks} contains remarks and future directions motivated by this work.

\section*{Acknowledgements} I wish to thank Alex Rutar for the interesting conversations regarding this note. I am also grateful to the anonymous referee, whose comments improved the exposition of the note.

\section{Background}\label{sec:background}

\subsection{Dimension notions.}
Let $(X,d)$ be a metric space. We use the Polish notation $|x-y|:=d(x,y)$ for all $x,y \in X$ and denote the open ball centered at $x$ of radius $r>0$ by $B(x,r):= \{z\in X: \,\, |x-z|<r \}.
$ Given a ball $B=B(x,r) \subset X$, we  denote by $\lambda B$ the ball $B(x,\lambda r)$, for $\lambda>0$. Given a subset $U$ of $X$, we denote by $|U|$ the diameter of $U$. While for every $x\in X$, $r>0$, we have $|B(x,r)|\leq 2r$, the radius of a ball in a metric space is not uniquely determined. However, for the rest of this manuscript we make the convention that for every ball $B$ that contains at least two distinct points, all radii of $B$ that we consider do not exceed $2|B|$. Thus, the notation $B(x,r)$ always implies that
\begin{equation}\label{eq: convention balls}
    2^{-1}|B(x,r)|\leq r\leq 2|B(x,r)|.
\end{equation}

We recall the dimension notions that we are focusing on in this note (see \cite{FalcBook, FraserBook} for more details). For $s>0$, $r>0$, and a subset $E$ of $X$, the \textit{$s$-dimensional $r$-approximate Hausdorff measure} of $E$ is defined as
$$
H_r^s(E)= \inf 
\left\{ \sum_i |U_i|^s : \{ U_i\}\text{\ is a countable cover of}
\,\,E \,\,\text{with} \,\, |U_i|\leq r  \right\}.
$$ The \textit{$s$-dimensional Hausdorff (outer) measure} of $E$ is the limit
$$
H^s(E)=\lim_{r \rightarrow 0} H_r^s(E),
$$ and the Hausdorff dimension of $E$ is defined as
$$
\dim_H E= \inf \{ s \geq 0 : H^s(E)=0 \}=\sup \{ s \geq 0 : H^s(E)=\infty \}.
$$

While the Hausdorff dimension stems from the idea of a dimensional threshold between null and infinite measure for the given set, other dimension notions rely on the behavior of the cardinality of coverings by small sets. Let $E$ be a bounded subset of $X$. For $r>0$, denote by $N(E,r)$ the smallest number of sets of diameter at most $r$ needed to cover $E$. The {\it (upper) Minkowski dimension} of $E$ is defined as
$$
\ovdimB E = \limsup_{r\to 0} \frac{\log N(E,r)}{\log(1/r)}.
$$
This notion is also known as \textit{upper box-counting dimension}, which justifies the notation with the subscript `B' typically used in the literature (see \cite{FalcBook}, \cite{FraserBook}). We drop the adjective `upper' and the bar notation throughout this paper as we will make no reference to the lower Minkowski dimension. For any fixed $r_0\leq |E|$, an equivalent formulation is
$$
\dim_B E = \inf \{\alpha>0 \,:\, \exists\,C>0\mbox{ s.t. } N(E,r) \le C r^{-\alpha} \mbox{ for all $0<r\le r_0$} \}.
$$ This formulation reveals the relation of the Minkowski dimension to other notions.

For an arbitrary (not necessarily bounded) set $E \subset X$, the {\it Assouad dimension} of $E$ is
$$
\dim_A E = \inf \left\{\alpha>0 \,:\, {\exists\,C>0\mbox{ s.t. } N(E\cap B(x,R),r) \le C (R/r)^{\alpha} \atop \mbox{ for all $0<r\le R$ and all $x \in E$}} \right\}.
$$
The idea behind the Assouad dimension is that we are allowed to zoom-in and zoom-out at all points of the set, trying to find the largest dimension possible.
This notion first appeared (under a different name) in a 1983 paper of Assouad on metric embedding problems \cite{Assouad83}. However, it has recently gained a lot of popularity in the fractal geometry and dynamics community, due to its various applications (see \cite{FraserBook} for a modern exposition).

It is easy to see that for a fixed bounded set $E\subset X$, the relation between the aforementioned dimensions is
$$
\dim_H E\leq \dim_B E\leq \dim_A E,
$$ with the inequalities being strict in various cases (see for instance \cite{ChronConcentric}). The potential ``gap" between the Minkowski and the Assouad dimension received increased attention by Fraser and Yu \cite{fy:assouad-spectrum}, who defined a collection of dimension notions that reside in that gap. For $0<\theta<1$ and a set $E \subset X$, define
$$
\dim_{A}^\theta E = \inf \left\{\alpha>0 \,:\, {\exists\,C>0\mbox{ s.t. } N(B(x,R) \cap E,r) \le C (R/r)^{\alpha} \atop \mbox{ for all $0<r\le R^{1/\theta}< R< 1$ and all $x \in E$}} \right\}.
$$
Thus, $\dim_{A}^\theta E$ is defined by the same process as $\dim_A E$, but with the restriction that the two scales $r$ and $R$ involved in the definition of the latter are related by the inequality $r^\theta\leq R$. The function $\theta\mapsto \dim_A^\theta E$ is called the {\it (regularized)\footnote{This definition can also be found as \textit{upper} Assouad spectrum in the literature, while the original spectrum in \cite{fy:assouad-spectrum} was defined with  $r=R^{1/\theta}$. Due to the relation $r\leq R^{1/\theta}$ ensuring that the spectrum is monotone in $\theta$ for a fixed $E$, thus resulting in a more regular function of $\theta$, the term ``regularized" was suggested by Tyson and the author. See also the discussion in \cite{OurQCspec}.} Assouad spectrum} of $E$. We often abuse the terminology and refer to a value $\dim_A^\theta E$ for some $\theta\in (0,1)$ as the Assouad spectrum of $E$. 

We emphasize that the Assouad spectrum naturally (and continuously \cite{FraserBook}) interpolates between the Minkowski and Assouad dimensions. For a fixed $E$, the limit $\lim_{\theta\to 0^+}\dim_{A}^\theta E$ exists and equals $\dim_B E$. Moreover, $\lim_{\theta\to 1^-}\dim_{A}^\theta E$ coincides with the so-called {\it quasi-Assouad dimension} of $E$, denoted by $\dim_{qA} E$. While there are instances where the quasi-Assouad dimension differs from the Assouad dimension, in many natural situations they coincide (see \cite{TroschQuasiAssouad3,TroschQuasiAssouad2,TroschQuasiAssouad1}). 
Furthermore, if $E$ is bounded, for all $\theta\in (0,1)$ we have 
$$
\dim_H E\leq \dim_B E\leq \dim_{A}^\theta E \le \dim_{qA}E \le \dim_A E,
$$ with the first two inequalities being strict in various cases (see for instance \cite[Theorem 3.4.7]{FraserBook} and \cite{ChronConcentric}).

\subsection{Metric dyadic cubes.}\label{subsec:MS}

In the context of the typical Euclidean metric spaces $\R^n$, it is quite elementary to equivalently express the dimensions $\dim_H E$ and $\dim_B E$ using dyadic cubes instead of arbitrary sets, by modifying the quantities $H^s_r(E)$ and $N(E,r)$. This idea was also extended  for the Assouad dimension and the Assouad spectrum (see \cite[Proposition 2.5]{OurQCspec} for a proof). In this subsection we present and define the appropriate dyadic cube notion in the context of metric spaces, in order to prove similar expressions for the aforementioned dimensions.

We say that $(X,d)$ is a \textit{doubling metric space} if there is a \textit{doubling constant} $C_d\geq 1$ such that for all $x\in X$ and $r>0$, the smallest number of balls of radius $r$ needed to cover $B(x,2r)$ is at most $C_d$.  The property of a metric space $X$ being doubling is in fact equivalent to $\dim_A X<\infty$ (see for instance \cite{hei:lectures}).
Note that the doubling property implies that $X$ is separable.

As noted in the Introduction, the dyadic cube systems that we employ are those constructed by Hyt\"onen and Kairema.
\begin{theoremA}[Hyt\"onen, Kairema \cite{Hyt:dyadic}]\label{thm:Dydadic}
	Suppose $(X,d)$ is a doubling metric space. Let $0<c_0\leq C_0<\infty$ and $\delta\in (0,1)$ with $12 C_0 \delta \leq c_0$. For any $k\in \N$ and collection of points $\{ z_i^k \}_{i\in I_k}$ with
	
	\begin{equation}\label{eq:centers_away}
		|z_i^k-z_j^k|\geq c_0 \de^k, \,\,\,\, \text{for} \,\, i\neq j,
	\end{equation}
	and
	\begin{equation}\label{eq:points_close_centers}
	\min_i |z_i^k-x|< C_0 \de^k, \,\,\,\, \text{for all} \,\, x\in X,
	\end{equation}
	where $I_k$ is an index set, we can construct a collection of sets $\{ Q_i^k \}_{i\in I_k}$ such that
	\begin{itemize}
		\item[(i)] if $l \geq k$ then for any $i\in I_k$, $j\in I_l$ either $Q_j^l\subset Q_i^k$ or $Q_j^l \cap Q_i^k=\emptyset$,
		\vspace{0.1cm}
		\item[(ii)] $X$ is equal to the disjoint union $\bigcup\limits_{i\in I_k} Q_i^k$, for every $k\in\N$,
		\vspace{0.1cm}
		\item [(iii)] $B(z_i^k, c_0 \de^k /3) \subset Q_i^k \subset B(z_i^k, 2C_0 \de^k)=:B(Q_i^k)$ for every $k\in \N$,
		\vspace{0.1cm}
		\item [(iv)] if $l\geq k$ and $Q_j^l\subset Q_i^k$, then $B(Q_j^l)\subset B(Q_i^k)$.
	\end{itemize}
	
	For $k\in \N$, we call the sets $Q_i^k$ from the construction of Theorem \ref{thm:Dydadic} ($\de$-)\textit{dyadic cubes} of level $k$ of $X$, and the collection $\{Q_i^k:k\in \N, i\in I_k\}$ a \textit{dyadic cube system}.
	
\end{theoremA}

Fix a doubling metric space $(X,d)$ with doubling constant $C_d$, and constants $\delta<1/100$, $c_0$ and $C_0$, as in Theorem \ref{thm:Dydadic}, for the rest of the paper. Moreover, for every $k\in \N$ we fix a collection of points $\{ z_i^k \}_{i\in I_k}$ satisfying \eqref{eq:centers_away}, \eqref{eq:points_close_centers}. 
To see why such a collection of points exists, consider the covering $\{ B(z,c_0 \delta^k): z\in X \}$ of $X$ and apply the $5B$-covering lemma (see \cite{hei:lectures}). By separability of $X$ and by choosing $c_0$ and $C_0$ so that $5c_0 \delta^k<C_0 \delta^k$, the existence of centers $\{ z_i^k \}_{i\in I_k}$ is ensured. 

By \cite[Theorem 4.1]{Hyt:dyadic}, we can fix finitely many dyadic cube systems of $X$ satisfying Theorem A, say $\mathcal{S}=\{\mathcal{D}_1, \dots, \mathcal{D}_K\}$ with $K=K(C_d,\delta)\in \N$, so that for every ball $B(x,R)$, $x\in X$, $R>0$, which contains at least two points, there are a constant $C_\delta>0$ and a $\de$-dyadic cube $Q(x,R)\in \mathcal{D}_t$, for some $t=t_{x,R}\in \{1,\dots, K\}$, so that $B(x,R)\subset Q(x,R)$ and
\begin{equation}\label{eq: cube corresp to ball}
    C_\de^{-1}R\leq |Q(x,R)|\leq C_\de R.
\end{equation} We call $Q(x,R)$ the \textit{circumscribed cube} of $B(x,R)$. This is the particular property of the dyadic systems from \cite{Hyt:dyadic} that ensures their compatibility with  Assouad-like dimensions. 

Henceforth, if two quantities $A,B>0$ are related by $A\leq C B$, for some uniform constant $C>0$ that depends only on intrinsic constants of $X$ and the systems in $\mathcal{S}$, such as $C_d, \de, c_0, C_0$, we write $A\lesssim B$ and call $C$ the \textit{comparability constant} of the relation. Similarly, if there is a uniform constant $C>0$ such that $A\geq C B$, we write $A\gtrsim B$. Lastly, if there are uniform constants $C_1,C_2>0$ such that $C_1 B\leq A\leq C_2 B$, we write $A\simeq B$. Note that the relation $A\simeq B$ is  equivalent to $C^{-1} B\leq A\leq C B$, for some uniform $C>0$, which we also call the comparability constant of the relation.

We have introduced all the necessary notions to rigorously define the cubic measure and dyadic cube covering number in the setting of metric spaces, which are stated in Theorem \ref{thm:main}. We further assume without loss of generality that all balls considered in what follows contain at least two points, even if not explicitly stated, since all covering number conditions in the definitions of the Assouad dimension and spectrum are trivial otherwise.
Given $x\in X$, $R>0$, $m\in \N$, denote by $W_m(Q(x,R))\subset \Dcal_t$ the dyadic cubes of level $L_{R}+m$ that are contained in $Q(x,R)\in \Dcal_t$, where $L_{R}$ is the level of the circumscribed cube $Q(x,R)$ of the ball $B(x,R)$. Note that by the convention $|B(x,R)|\geq R/2$, Theorem \ref{thm:Dydadic} (iii) and \eqref{eq: cube corresp to ball}, the level $L_R$ of $Q(x,R)$ does not depend on the point $x$. Given a set $E\subset X$, we denote by $D(E\cap B(x,R),m)$ the smallest number of cubes in $W_m(Q(x,R))$ needed to cover $E\cap B(x,R)$. If $E\subset B(x,R)$, we simplify the notation and write $D(E,m)$.

Given $\Dcal\in \mathcal{S}$, $s>0$, $r>0$ and $E\subset X$, the \textit{$s$-dimensional $r$-approximate cubic measure} of $E$ with respect to $\Dcal$ is defined as
\begin{equation}\label{eq: def of Mrs}
    M_r^s(E)=\inf \left\{ \sum_i |Q_i|^s : \parbox{65mm}{\raggedright$\{ Q_i\}\subset \Dcal$ is a cover of
$E$ by dyadic cubes of level $m$ with $4C_0 \delta^m\leq r$}  \right\}.
\end{equation}
The \textit{$s$-dimensional cubic measure} of $E$ with respect to $\Dcal$ is then defined as
$$
M^s(E)=\lim_{r\rightarrow0} M_r^s(E).
$$ The fact that $M_r^s(E)$ is decreasing in $r$ ensures that the above limit exists.

\section{Proof of equivalent expressions}\label{sec: Proofs}

\subsection{Hausdorff dimension and cubic measure}\label{subsec: dimH}
In this subsection we show that the threshold dimension of the cubic measure $M^s(E)$ is indeed equal to the Hausdorff dimension $\dim_H E$.

\begin{proof}[Proof of Theorem \ref{thm:main} (i)]
    Let  $s>0$, $r\in (0,1)$, and $E\subset X$. Fix a dyadic system $\Dcal\in \mathcal{S}$ of $X$, and consider the cubic measure $M_r^s(E)$ with respect to $\Dcal$. We first show that $M_r^s(E)\simeq H_r^s(E)$, with comparability constant independent of $r$. 

    One direction is trivial, namely, $H_r^s(E)\leq M_r^s(E)$, due to the cubes used in the definition of $M_r^s(E)$ having diameter at most $r$, by Theorem \ref{thm:Dydadic} (iii). For the other direction, let $\{U_i\}$ be a cover of $E$ with $|U_i|\leq r$. For every $i$, fix some $x_i\in U_i$ and set $m_i$ to be the unique integer such that
    \begin{equation}\label{eq: dimH Ui mi compare}
        4C_0 \delta^{m_i+1}< |U_i|\leq 4C_0\delta^{m_i}.
    \end{equation}
    Note that we assume that $|U_i|>0$, otherwise it would not contribute to the sum. Due to \eqref{eq:centers_away} and the one-to-one correspondence of cubes of level $m_i$ and centers $z_\ell^{m_i}$, along with $\dim_A X<\infty$ imply that there can be at most $C'>0$ cubes of level $m_i$ intersecting the ball  $B(x_i,4C_0\delta^{m_i})$, for some constant $C'$ that depends only on $\delta$ and the doubling constant $C_d$ of $X$. By \eqref{eq:points_close_centers}, these cubes in fact cover the ball $B(x_i,4C_0\delta^{m_i})$.  An additional application of the doubling condition of $X$ yields that every dyadic cube of level $m_i$ intersecting $B(x_i,4C_0\delta^{m_i})$ can be decomposed into at most $C''>0$ cubes of level $m_i+1$, with $C''$ depending only on $\delta$ and $C_d$. Denote this collection of  cubes of level $m_i+1$ by $\{Q_j^{m_i+1}\}_{j}$. Note that by the right-hand side of \eqref{eq: dimH Ui mi compare} we have $U_i\subset B(x_i,4C_0\delta^{m_i})$, for every $i$. Thus, since $\{U_i\}_i$ is a cover of $E$, we have
    $$
    E\subset \bigcup_{i,j}Q_j^{m_i+1},
    $$ with $|Q_j^{m_i+1}|\leq4C_0 \delta^{m_i+1}\leq r$ by Theorem \ref{thm:Dydadic} (iii), \eqref{eq: dimH Ui mi compare}, and choice of $\{U_i\}$. In addition, by \eqref{eq: dimH Ui mi compare} we have
    $$
    \sum_{i,j}|Q_j^{m_i+1}|^s\leq \sum_{i,j}(4C_0 \delta^{m_i+1})^s\leq C''\sum_{i}(4C_0 \delta^{m_i+1})^s\lesssim \sum_i |U_i|^s,
    $$ where the comparability constant does not depend on $r$. Since the collection of cubes $\{Q_j^{m_i+1}\}_{i,j}$ satisfies the conditions for a cube cover of $E$ in \eqref{eq: def of Mrs}, the above inequality implies
    $$
    M_r^s(E)\lesssim \sum_i |U_i|^s.
    $$ But the cover $\{U_i\}$ is arbitrary, which implies by the above that
    $$
    M_r^s(E)\lesssim H_r^s(E)
    $$ as needed. As a result,  it is shown that
    $$
    M_r^s(E)\simeq H_r^s(E),
    $$ with the comparability constant independent of $r$, which allows for $r\rightarrow 0$, proving that $M^s(E)\simeq H^s(E)$. By the definition of $\dim_H E$, this is enough to complete the proof.

\end{proof}

\subsection{Dyadic covering numbers and dimensions}\label{subsec: dimA}

Recall that for a fixed, but arbitrary dyadic cube system of $X$, the author  proved in \cite{Chron_Sob} that using dyadic cubes instead of arbitrary sets indeed yields the Minkowski dimension.
\begin{proof}[Proof of Theorem \ref{thm:main} (ii)]
    Suppose $E\subset X$ is bounded, with $E\subset B(x,R)$, for some $x\in X, R>0$. Let $\Dcal_t\in \mathcal{S}$ be the dyadic cube system in which the circumscribed cube $Q(x,R)$ lies, and fix $m_E\in \N$ so that $\delta^{m_E}\leq |E|$. Thus, for any integer $m\geq m_E$, the covering number $D(E,m)$ is simply the minimum number of cubes of level $L_{R}+m$ in $\Dcal_t$ needed to cover $E$. The desired relation for $\dim_B E$ follows by \cite[Proposition 2.1]{Chron_Sob} (where $D(E,m)$ essentially corresponds to the covering number $N_{L_{R}+m}(E)$ in the notation of Proposition 2.1, for the fixed given dyadic cube system $\Dcal_t$).
\end{proof}

\begin{remarks}~
\begin{enumerate}[label=(\roman*)]
    \item After a closer analysis of the above proof, it is evident that the statement of Theorem \ref{thm:main} (ii) can be slightly strengthened by including potentially more levels of cubes in the expression of $\dim_B E$. In particular, the number $m_E$ can be replaced by the minimal integer so that $\delta^{L_{R}+m_E}\leq |E|$, and the statement still holds.
    \item It should be noted that in the proofs of Theorem \ref{thm:main} (i) and (ii), the fact that the dyadic system $\mathcal{D}$ is one of those lying in $\mathcal{S}$ is not used. Hence, the Hausdorff and Minkowski dimensions can in fact be expressed using \textit{any} $\delta$-dyadic system with the properties listed in Theorem \ref{thm:Dydadic}.
\end{enumerate}
\end{remarks}

We proceed with the proofs regarding the Assouad spectrum and the Assouad dimension.

\begin{proof}[Proof of Theorem \ref{thm:main} (iii)]
    We first pick the desired uniform constant $\Ctil>0$ to be
    \begin{equation}\label{eq: Ctilde dimAtheta}
        \Ctil=\frac{12 C_0 C_\delta}{c_0},
    \end{equation} where $c_0, C_0$ are the constants as in Theorem \ref{thm:Dydadic} and $C_\delta$ is the uniform constant in \eqref{eq: cube corresp to ball}.
    Let $\theta\in (0,1)$, and set
    $$
    A_\theta= \left\{\alpha>0 \,:\, {\exists\,C>0\mbox{ s.t. } N(B(x,R) \cap E,r) \le C (R/r)^{\alpha} \atop \mbox{ for all $0<r\le R^{1/\theta}< R< 1$ and all $x \in E$}} \right\},
    $$ so that $\dim_A^\theta E=\inf A_\theta$ by definition, and
    $$
    B_\theta= \left\{\alpha>0 \,:\, {\exists\,C>0\mbox{ s.t. } D(E\cap B(x,R),m) \le C \delta^{-m\alpha} \atop \mbox{ for all $x\in E, m\in \N, R>0$ with $0<\Ctil \delta^{m}R\le R^{1/\theta} < R < 1$}} \right\},
    $$ for which we need to show that $\dim_A^\theta E=\inf B_\theta$. We do so by proving that in fact $A_\theta=B_\theta$. 
    We also emphasize that even if we temporarily consider balls $B=B(x,R)$ in the definitions of the sets $A_\theta, B_\theta$ against the convention \eqref{eq: convention balls}, i.e.~with $R>2|B(x,R)|$, then due to $N(B(x,R)\cap E,r)=N(B(x,2|B|)\cap E,r)$ the proof still reduces to considering only balls according to the convention.
    
    Let $x\in E$, $R>0$, $m\in \N$ and a cube $Q_i\in W_m(Q(x,R))\in \Dcal$, for some $\Dcal\in \mathcal{S}$ where the circumscribed cube $Q(x,R)$ is contained. By definition of $W_m(Q(x,R))$, the cube $Q_i$ is of level $L_R+m$, where $L_R$ is the level of $Q(x,R)$. Since $|Q(x,R)|\geq 3^{-1}c_0 \delta^{L_R}$ by Theorem \ref{thm:Dydadic} (iii) and \eqref{eq: convention balls}, we have by \eqref{eq: cube corresp to ball} and another application of Theorem \ref{thm:Dydadic} (iii) on $Q_i$ that
    \begin{equation}\label{eq: diamQi at most Ctil}
        |Q_i|\leq 4C_0 \delta^{L_R+m}=4C_0 \delta^m \delta^{L_R}\leq 4C_0 \delta^m 3 c_0^{-1}|Q(x,R)|\leq \Ctil \delta^m R.
    \end{equation} Hence, the cubes used in the definition of $B_\theta$, i.e., cubes in $W_m(Q(x,R))$ for $m\in\N$ such that $\Ctil \delta^m R\leq R^{1/\theta}$, are the appropriate cubes to use in coverings of $E$ for the Assouad spectrum $\dim_A^\theta E$, since the diameter of these cubes are at most $R^{1/\theta}$ by \eqref{eq: diamQi at most Ctil}. This trivially leads to the relation 
    \begin{equation}\label{eq: dimA N less D}
        N(B(x,R)\cap E,\Ctil \delta^m R)\leq D(E\cap B(x,R),m),
    \end{equation} for all $m\in\N$ such that $\Ctil \delta^m R\leq R^{1/\theta}$.

    Let $\alpha\in B_\theta$, and $0<r\leq R^{1/\theta}$. Set $m_r\in \N$ to be the unique integer so that
    $$
    \Ctil \delta^{m_r}R\leq r<\Ctil \delta^{m_r-1}R.
    $$ Since the cubes in $W_{m_r}(Q(x,R))$ are of diameter at most $\Ctil\delta^{m_r}R\leq r\leq  R^{1/\theta}$,  we have by \eqref{eq: dimA N less D} that
    $$
    N(B(x,R)\cap E, r)\leq N(B(x,R)\cap E,\Ctil \delta^{m_r} R)\leq D(E\cap B(x,R),m_r).
    $$ In addition, by $\alpha\in B_\theta$ and by choice of $m_r$, the above implies
    $$
    N(B(x,R)\cap E, r) \leq C\delta^{-m_r \alpha}\leq (\Ctil \delta^{-1})^\alpha \left( \frac{R}{r} \right)^\alpha.
    $$ Since $x\in E$ and the scales $r, R$ are arbitrary, we have $\alpha\in A_\theta$, which proves the inclusion $B_\theta \subset A_\theta$.

    For the other inclusion we need to show a comparability relation similar to \eqref{eq: dimA N less D} with opposite direction. Let $\{U_i\}_i$ be a cover of $E$ with $|U_i|\leq \Ctil \delta^m R$.  Recall by \eqref{eq: cube corresp to ball} that $R\simeq \delta^{L_R}$, which implies $|U_i|\lesssim \delta^{L_R+m}$. By a similar argument to that in the proof of Theorem \ref{thm:main} (i), it can be shown that there are at most $M_0$ cubes of level $L_R+m$ needed to cover $U_i$, for some uniform constant $M_0$ that does not depend on $i$, $m$, $R$.
    %
    %
    This establishes the relation
    \begin{equation}\label{eq: dimA D lesssim N}
        D(E\cap B(x,R),m)\leq M_0 N(B(x,R)\cap E,\Ctil \delta^m R),
    \end{equation} for all $m\in\N$ such that $\Ctil \delta^m R\leq R^{1/\theta}$.
    
    Let $\alpha\in A_\theta$, and $m\in\N$ with $\Ctil \delta^m R\leq R^{1/\theta}$. For $r=\Ctil \delta^m R\leq R^{1/\theta}$, by \eqref{eq: dimA D lesssim N} and definition of $A_\theta$ we have 
    $$
    D(E\cap B(x,R),m)\leq M_0 N(B(x,R)\cap E,r)\leq M_0 C \left( \frac{R}{\Ctil \delta^m R}\right)^\alpha= M_0 C \Ctil^{-\alpha} \delta^{-m\alpha},
    $$ which implies that $\alpha\in B_\theta$. Therefore, $A_\theta=B_\theta$, and the proof is complete.

    (iv) While the statement does not follow by (iii), due to the limit of $\dim_A^\theta E$ as $\theta\rightarrow 1^-$ being equal to the quasi-Assouad dimension $\dim_{qA} E\leq \dim_A E$, the proof is nonetheless identical to that of (iii), by replacing $\theta$ with the number $1$ in all arguments.
\end{proof}

\section{Final Remarks}\label{sec:Final Remarks}

The construction of dyadic cube systems of Hyt\"onen-Kairema in \cite{Hyt:dyadic} is actually given for quasimetric spaces. As a result, Theorem \ref{thm:main} is also true if $X$ is a quasimetric doubling space. The proofs are almost identical, with the only difference being the dependence of a few of the uniform constants on the quasimetric constant of the space. Another generalization of Theorem \ref{thm:main} would be to state it for all $\delta\in (0,1)$ small enough that satisfy the relation $12C_0 \delta\leq c_0$ in Theorem \ref{thm:Dydadic}, which follows with identical arguments to those in Section \ref{sec: Proofs}.

While the  cube systems defined in Theorem \ref{thm:Dydadic} have been used extensively in various areas \cite{HytAppl1, HytAppl2, HytAppl3}, there are many different constructions of dyadic cubes in the metric setting. The main advantage of the Hyt\"onen-Kairema systems is the existence of a circumscribed cube for every ball in $X$. This is a fundamental property that is necessary for the arguments in the proof of Theorem \ref{thm:main} (iii), (iv). This is evident already by the representation of the Assouad dimension and spectrum using dyadic cubes in $\R^n$, as shown in \cite[Proposition 2.5]{OurQCspec}. It would be interesting to establish similar expressions for the dimensions in Theorem \ref{thm:main} using a different construction of dyadic cube systems on metric spaces, for instance the notion defined in \cite{MoreCubes2}. We expect the corresponding proofs for the Hausdorff and Minkowski dimension to be almost identical to those in Section \ref{sec: Proofs}. However, the Assouad dimension and Assouad spectrum need to be treated differently, due to the lack of circumscribed cubes for arbitrary balls.

Moreover, Falconer, Fraser and Kempton introduced in \cite{IntermDimensions} the \textit{intermediate \mbox{dimensions}}, a collection of dimensions that interpolate between the Hausdorff and Minkowski dimension, similarly to how the Assouad spectrum interpolates between the Minkow\-ski and Assouad dimension. It would be an interesting result to represent those dimensions using dyadic cube systems in the metric setting as well. Since a circumscribed cube is only necessary in the arguments in Section \ref{sec: Proofs} due to the local ``zooming in" nature of the Assouad dimension and spectrum, we expect to have similar representations to those in Theorem \ref{thm:main} for the intermediate dimensions, using any of the systems from \cite{Hyt:dyadic, MoreCubes2}


\end{document}